 \documentclass[11pt]{article}

\usepackage{amsmath,amsfonts,amsthm,amscd,amssymb,graphicx}
\usepackage{mathabx}

\usepackage[utf8]{inputenc}

\usepackage{subfigure}
\usepackage{xcolor}

\numberwithin{equation}{section}

\usepackage{hyperref}

\usepackage{cases}

\newtheorem{theorem}{Theorem}[section]

\newtheorem{lemma}[theorem]{Lemma}

\newtheorem{remark}[theorem]{Remark}

\newcommand{\RR}{\mathbb{R}}

\newcommand{\cH}{\mathcal{H}}

\def\cA{\mathcal{A}}

\begin{document}

\title{Modified scattering for long-range Hartree equations of infinite rank near vacuum} 

\author{Toan T. Nguyen\footnotemark[1]
	\and Chanjin You\footnotemark[1]
}

\maketitle

\footnotetext[1]{Penn State University, Department of Mathematics, State College, PA 16802. Emails: nguyen@math.psu.edu, cby5175@psu.edu. }


\begin{abstract}

We establish the asymptotic behavior and decay of solutions near vacuum to the Hartree equation with the Coulomb interaction potential in three dimensions. Our approach is direct, which consists of independently deriving the sharp dispersive decay estimates for the density function and establishing the boundedness of energy norms. The global in time well-posedness is done without introducing the phase correction, while the phase modification is explicit in terms of the density function. 

\end{abstract}


\section{Introduction}
Consider the time-dependent Hartree equation of infinite rank
\begin{equation}
	\label{Hartree}
	\begin{cases}
		\begin{aligned}
			&i \partial_{t} \gamma = [ - \Delta + w \star_x \rho_\gamma , \, \gamma ] \\
			&\gamma_{\vert_{t=0}} = \gamma_{0}
		\end{aligned}
	\end{cases}
\end{equation}
in the whole space $\mathbb{R}^3$.
Here $\gamma(t)$ is the one-particle density operator, which is a self-adjoint, nonnegative and bounded operator on $L^2(\mathbb{R}^3)$. The nonlinear interaction between particles are described through the spatial convolution $w \star_{x} \rho_{\gamma}$, where $w$ is the two-body interaction potential, and $\rho_{\gamma}(t)$ is the density function associated with operator $\gamma(t)$ defined by 
$$\rho_{\gamma}(t,  x) = \gamma(t,x,x)$$ 
where $\gamma(t,x,y)$ is the integral kernel of $\gamma(t)$. 


In the case when $\gamma(t)$ is of rank $1$, there is some $u \in L^2 (\mathbb{R}^{d})$ so that $\gamma(t) = \lvert u(t) \rangle \langle u(t) \rvert$. Then \eqref{Hartree} is reduced to the standard Hartree equation, namely,
\begin{equation}
	\begin{cases}
		\begin{aligned}
			&i \partial_{t} u = (- \Delta + w \star_x |u|^2) u \\
			&u_{\vert_{t=0}} = u_{0}.
		\end{aligned}
	\end{cases}
\end{equation}
There are vast literatures concerning the large time behavior of solutions to this equation. Especially, for $d=3$, Hayashi and Naumkin \cite{Hayashi1998a} proved the global wellposedness and modified scattering using the pseudoconformal invariance of free Schr\"{o}dinger operator. Later the same result is reproved by Kato and Pusateri \cite{Kato2011}, using the spacetime resonance method. In the case when $\gamma(t)$ is of finite rank, the scattering theory was established in \cite{Wada2002, Ikeda2012}.

Our main interest is \eqref{Hartree} in which infinite rank operators are allowed. The global wellposedness of \eqref{Hartree} for trace-class initial data was studied in \cite{Bove1974, Bove1976, Chadam1976, Zagatti1992, Jerome2015}. Pusateri and Sigal \cite{Pusateri2021} studied the asymptotic stability of small solutions of the time-dependent Kohn-Sham equations. {Their results include scattering for Hartree equations in the full short-range regime, leaving the Coulomb case open. We resolve it in this paper, confirming the Coulomb potential is critical to the scattering theory.}

Furthermore, we propose a direct simple proof of the global wellposedness of \eqref{Hartree} with the Coulomb interaction potential in $\mathbb{R}^3$, namely, $w(x) =\pm |x|^{-1}$, $x \in \mathbb{R}^3$. This also shows that the time decay of global solutions is the same as that of linear solutions. 
{Interestingly, phase correction is not necessary for the global wellposedness.}
Next, we investigate the asymptotic stability of small solutions of \eqref{Hartree}. We establish a modified scattering result for \eqref{Hartree}.
{We construct a phase correction term at the level of the density function, rather than the density operator itself. This allows us to extract the term decaying at rate $t^{-1}$ without performing the entire spacetime resonance argument (e.g., \cite{Kato2011}).}

\subsection{Main result}

Our main result concerns with the scattering of solutions to the Cauchy problem \eqref{Hartree} with spatially localized data. We denote by $\cH^k$ and $\cA^k$ the spaces of self-adjoint, nonnegative and bounded operators $\gamma$ on $L^2(\mathbb{R}^3)$
with corresponding finite norms
\begin{equation}\label{def-norms} 
\begin{aligned}
\| \gamma \|_{\cH^k} &= \| \langle \nabla_x\rangle^k \langle \nabla_y\rangle^k \gamma(\cdot, \cdot)\|_{L^2_{x,y}}
, \qquad \| \gamma \|_{\cA^k} = \| \langle x\rangle^k \langle y\rangle^k \gamma(\cdot, \cdot)\|_{L^2_{x,y}}
\end{aligned}\end{equation}
respectively, in which $\gamma(x,y)$ denotes the integral kernel of $\gamma$. Here, $\langle a\rangle = \sqrt{1+|a|^2}$. When $k=0$, we simply write $\mathcal{HS} =\cH^0= \cA^0$ (which coincides with the usual Hilbert-Schmidt spaces). 

The main result of this paper reads as follows. 

\begin{theorem}
	\label{thm:GWP}
Let $w(x) = \pm |x|^{-1}$ be the Coulomb pair interaction potential. Assume that initial data $\gamma_{0}$ is of trace-class with finite norm 
	$$\varepsilon :=\| \gamma_{0}\|_{\cA^2} + \|\gamma_{0} \|_{\cH^2}  < \infty.$$
Then, for sufficiently small $\varepsilon $, the Cauchy problem \eqref{Hartree} has a unique global-in-time solution $ \gamma \in C(\RR_+; \cH^2)$, whose density satisfies   
	\begin{equation}\label{eq:rho-decay}
		\|  \rho (t) \|_{L^{p}_{x}} \lesssim \varepsilon  \langle t \rangle^{- 3(1-1/p)} 
	\end{equation}
	for $p\in [1, \infty]$. In addition, there is a unique operator $\gamma_{\infty} \in \mathcal{HS}$ and function $g_\infty(\cdot)$ so that 
		\begin{equation}\label{modifieds}
			\|   \gamma(t) - e^{-i(-t\Delta + g_{\infty}(-i\nabla) \log t)} \gamma_{\infty} e^{i(-t\Delta + g_{\infty}(-i\nabla) \log t)} \|_{\mathcal{HS}} \lesssim \varepsilon  \langle t \rangle^{-\delta}. 
		\end{equation}
\end{theorem}

Theorem \ref{thm:GWP} establishes the modified scattering of solutions to \eqref{Hartree} for small and localized trace-class data, and for the critical Coulomb interaction potential $w(x) = \pm|x|^{-1}$. The proof of Theorem \ref{thm:GWP} consists of deriving the sharp decay for the potential function $V(t,x) = \pm|x|^{-1}\star_x\rho(t,x)$ and standard energy estimates on $\gamma$ and its conjugate operator $e^{-it \Delta} \gamma e^{it \Delta}$. Our proof is rather simple and direct. Specifically, fix a small constant $\delta>0$, and for any $t\ge 0$, introduce the following iterative norm 
\begin{equation}\label{def-normX}
	\| \gamma\|_{X_t}= \sup_{0\le s\le t} \Big\{ \|\rho(s)\|_{L^1}+
\langle s \rangle^{3} \| \rho(s)\|_{L^{\infty}} 
		+ \langle s \rangle^{-\delta} \| \gamma(s) \|_{\cH^{2}}
		+ \langle s \rangle^{-\delta} \|e^{-is \Delta} \gamma(s) e^{is \Delta}\|_{\cA^2}
	\Big\},
\end{equation}
for some fixed positive constant $\delta$. By standard local existence theory, $\| \gamma\|_{X_t}$ exists and satisfies $\| \gamma\|_{X_t}\le C_0\varepsilon $ for some sufficiently small time $t>0$. In order to propagate the iterative norm globally in time, it suffices to derive the following a priori bound 
\begin{equation}\label{apriori}
\| \gamma\|_{X_t} \le C_0 \varepsilon  + C_1 \| \gamma\|_{X_t}^2
\end{equation}
for any $t\ge 0$ and for some universal constants $C_0, C_1$ that are independent of $t\ge 0$. From \eqref{apriori}, the global existence theory follows. Indeed, let 
 $T_*$ be the maximal time of existence and set 
\[
	T= \sup \Big\{ t\in (0,T_*):~\| \gamma\|_{X_t}\le 2C_0\varepsilon  \Big\}.
\]
Suppose that $T<T_*$. Then,  
\[
	2C_0\varepsilon   = \| \gamma\|_{X_T} \le C_0\varepsilon   + C_1 (2C_0\varepsilon  )^2 \le \frac{3}{2} C_0\varepsilon 
\]
for sufficiently small $\varepsilon $ so that $4C_0C_1 \varepsilon  <1/2$. The above yields a contradiction. Therefore, $T=T^{\ast}$ and $T^{\ast} = \infty$, since otherwise the local wellposedness theory can be applied to go past the time $T=T_*$. This proves that $\| \gamma\|_{X_t}\le 2C_0\varepsilon$ for all times $t\ge 0$. The global existence and decay estimates thus follow. It thus remains to establish the a priori estimates \eqref{apriori}. 
For the sake of simplicity and clarity, we made no attempts to optimize the decay rate in \eqref{modifieds}, for which one would need to derive phase mixing estimates for the density, similar to what is established in the companion paper \cite{Chanjin}. 

The rest of the paper is outlined as follows. The decay estimates for the density will be done in Section \ref{sec-decay}, while energy estimates are carried out in Section \ref{sec-EE}. Finally, the scattering profiles will be constructed in Section \ref{subsec:linfty}. 

\section{Density estimates}\label{sec-decay}

\begin{lemma} Let $\rho(t,x)$ be the density of $\gamma(t)$, and set $\mu = e^{-it \Delta} \gamma e^{it \Delta}$. Then, there holds 
\begin{equation}\label{denxy}
	\rho(t,x) = \frac{c_0}{t^3} \iint \exp \left( \frac{i}{4t}|x-y|^2 - \frac{i}{4t} |x-z|^2 \right) \mu(t,y,z) dy dz.
\end{equation}
for some constant $c_0$. 
\end{lemma}
\begin{proof} By definition, $\gamma = e^{it \Delta} \mu e^{-it \Delta}$, and so 
$$\begin{aligned} 
\gamma(t,x,y) 
&=
 e^{it (\Delta_x - \Delta_y)} \mu(t,x,y) 
 \\
 & = 
 \frac{c_0}{t^3} \iint \exp \left( \frac{i}{4t}|x-x'|^2 - \frac{i}{4t} |y - y'|^2 \right) \mu(t,x',y') \; dx'dy' 
 \end{aligned}$$
in which we recall that $t^{-3/2}e^{i|x-x'|^2/4t}$ is the integral kernel of $e^{it \Delta_x}$. 
The lemma thus follows.
\end{proof}

\begin{lemma}\label{lemma:linfty} Let $\rho(t,x)$ be the density of $\gamma(t)$, and set $\mu = e^{-it \Delta} \gamma e^{it \Delta}$.
Then, for $\beta \ge 0$, there holds 
	\[
		\|  \rho(t) \|_{L^{\infty}_x} \lesssim \langle t\rangle^{-3} \| \widehat{\mu}(t) \|_{L^{\infty}_{k,p}} + \langle t\rangle^{-3-\beta} \|\langle x\rangle^m \langle y\rangle^m\mu(t)\|_{L^2_{x,y}},
	\]
	provided that $m > \frac{3}{2} + \beta$.
\end{lemma}

\begin{proof} We consider the case when $t\ge 1$. Recalling \eqref{denxy}, we write 
$$ e^{\frac{i}{4t}|x-y|^2 - \frac{i}{4t}|x-z|^2 }=e^{ - \frac{i}{2t} x \cdot (y-z)} +  e^{ - \frac{i}{2t} x \cdot (y-z)}\Big( e^{\frac{i}{4t}(|y|^2 - |z|^2)} - 1 \Big)$$
and so 
$$\begin{aligned}
	\rho(t,x) &= \frac{c_0}{t^3} \iint e^{ - \frac{i}{2t} x \cdot (y-z)} \mu(t,y,z) dy dz + \frac{c_0}{t^3} \iint e^{ - \frac{i}{2t} x \cdot (y-z)}\Big( e^{\frac{i}{4t}(|y|^2 - |z|^2)} - 1 \Big)\mu(t,y,z) dy dz.
\end{aligned}$$
By definition, the first integral term is equal to 
\begin{equation}\label{den-ld}  \frac{c_0}{t^3} \iint e^{ - \frac{i}{2t} x \cdot (y-z)} \mu(t,y,z) dy dz = \frac{1}{t^3} \widehat{\mu}(t,\frac{x}{2t}, -\frac{x}{2t})
\end{equation}
where $\widehat{\mu}(t,k,p)$ denotes the Fourier transform of $\mu(t,x,y)$ with respect to variables $x,y$. On the other hand, using
\[
		\left|e^{\frac{i}{4t}(|y|^2 - |z|^2)} - 1 \right| 
				\lesssim \frac{(|y|^2 + |z|^2)^{\beta}}{t^{\beta}}
	\]
	for any $\beta >0$, we get
	\begin{align*}
		\Big| \frac{c_0}{t^3} \iint e^{ - \frac{i}{2t} x \cdot (y-z)}\Big( e^{\frac{i}{4t}(|y|^2 - |z|^2)} - 1 \Big)\mu(t,y,z) dy dz\Big|
		&\lesssim \frac{1}{t^{3+\beta}} \|(|y|^2 +|z|^2)^{\beta} \mu(t)\|_{L^1_{y,z}},
		\end{align*}
	which is bounded by $\frac{1}{t^{3+\beta}} \| \mu(t)\|_{H^{0,m}_{y,z}}$
	for $m > \frac{3}{2} + \beta$. This proves the desired estimates for $\rho(t,x)$. 
\end{proof}

\begin{lemma}\label{lem-Vde}
	Let $V =\pm |x|^{-1}\star \rho(t,x)$. As long as $\|\rho(t)\|_{L^p_x} \le \langle t\rangle^{-3(1-1/p)}\|\gamma\|_{X_{t}}$ for $p\in [1,\infty]$, there hold
	$$\| \partial^\alpha_x V(t)\|_{L^\infty_x} \lesssim \langle t\rangle^{-1-|\alpha|} \|\gamma\|_{X_{t}} ,$$
	for $0\le |\alpha|\le 1$. In addition, 
	$$\|\partial_x^2 V(t)\|_{L^p_x} \lesssim \langle t\rangle^{-3(1-1/p)}\|\gamma\|_{X_{t}},$$
	for $1<p<\infty$. 
\end{lemma}
\begin{proof}
	By definition, we compute 
	$$
	\begin{aligned}
		|V(t,x)| 
		&\le \int_{|x-y|\le t}\frac{1}{|x-y|}|\rho(t,y)|\; dy + \int_{|x-y|\ge t}\frac{1}{|x-y|}|\rho(t,y)|\; dy
		\\
		&\le \|\rho(t)\|_{L^\infty_x}\int_{|x-y|\le t}\frac{1}{|x-y|}\; dy + \langle t\rangle^{-1}\int_{|x-y|\ge t} |\rho(t,y)|\; dy
		\\&\lesssim \langle t\rangle^{-1} \|\gamma\|_{X_{T}},\end{aligned}$$ 
	since $\|\rho(t)\|_{L^p_x} \le \langle t\rangle^{-3(1-1/p)}\|\gamma\|_{X_{T}}$ by definition of the $\|\gamma\|_{X_{T}}$ norm. The bounds on $\partial_x V$ follow similarly. Finally, 
	the last estimate follows from the fact that $\partial_x^2 \Delta_x^{-1}$ is a Calderon-Zygmund operator and thus bounded on $L^p$ for $1< p < \infty$. 
\end{proof}
\section{Energy estimates}\label{sec-EE}

\begin{lemma}\label{lem-EE} Let $\delta>0$ be as in \eqref{def-normX}. 
For any $0\le t\le T$, there holds
\[
\frac{d}{dt}\| \partial_{x}^{\alpha} \partial_{y}^{\beta} \gamma(t) \|_{L^2_{x,y}} \lesssim \| \gamma \|_{X_{T}}^{2} \langle t\rangle^{-1+\delta} 
\]
for any $|\alpha|, |\beta|\le 2$. 
\end{lemma}
\begin{proof}
Let $\gamma(t,x,y)$ be the integral kernel of $\gamma(t)$. Then
\begin{equation}\label{eqs-gamma}
	i \partial_{t} \gamma(t,x,y) = (-\Delta_{x} + \Delta_{y}) \gamma(t,x,y) + (V(t,x) - V(t,y)) \gamma(t,x,y)
\end{equation}
where $V = w \star_{x} \rho$. For any $\alpha, \beta$, we compute 
$$ 
\frac{1}{2}\frac{d}{dt} \| \partial_{x}^{\alpha} \partial_{y}^{\beta} \gamma(t) \|_{L^2_{x,y}}^2 
=  \Im \iint \partial_{x}^{\alpha} \partial_{y}^{\beta}\Big((V(t,x) - V(t,y))\gamma(t,x,y)\Big) \partial_{x}^{\alpha} \partial_{y}^{\beta}\widebar\gamma(t,x,y)  \; dxdy.
$$ 
Hence, for $|\alpha|, |\beta|\le 2$, we bound 
\begin{equation*}
\begin{aligned}
	\frac{d}{dt}\| \partial_{x}^{\alpha} \partial_{y}^{\beta} \gamma(t) \|_{L^2_{x,y}} 
	&\le \Big\| \partial_{x}^{\alpha} \partial_{y}^{\beta}\Big((V(t,x) - V(t,y))\gamma(t)\Big) \Big\|_{L^2_{x,y}}  
\\&\le \| V(t)\|_{W^{1,\infty}_x} \| \gamma(t)\|_{H^2_xH^2_y} + \| \partial_x^\alpha V(t)\|_{L^2_x} \| \partial_y^\beta\gamma(t)\|_{L^\infty_xL^2_y}+ \| \partial_x^\beta V(t)\|_{L^2_x} \| \partial_y^\alpha\gamma(t)\|_{L^\infty_yL^2_x}
\end{aligned}
\end{equation*}
in which the last two terms are only present when $|\alpha|=2$ or $|\beta|=2$. Recall $\|\gamma(t)\|_{\cH^2} = \| \gamma(t)\|_{H^2_xH^2_y}$. Using Lemma \ref{lem-Vde} and the definition of $\| \gamma\|_{X_T}$ norm, we bound  
\[
\| V(t)\|_{W^{1,\infty}_x} \| \gamma(t)\|_{\cH^2} 
	\lesssim \| \gamma\|_{X_{T}}^{2} \langle t\rangle^{-1+\delta} .
	\]
Using again Lemma \ref{lem-Vde} together with the Sobolev embedding $H^2 \subset L^\infty$, we get 
\[
\| \partial_x^2 V(t)\|_{L^2_x} \| \partial_y^\beta\gamma(t)\|_{L^\infty_xL^2_y}\lesssim \langle t\rangle^{-3/2}\| \gamma\|_{X_{T}} \| \gamma(t)\|_{\cH^2}\lesssim \langle t\rangle^{-3/2+\delta} \|\gamma\|_{X_{T}}^2.
\]
The lemma thus follows. 
\end{proof}


\begin{lemma}\label{lem-bdmu} Let $\delta>0$ be as in \eqref{def-normX}. 
For any $0\le t\le T$, there holds
\[
\frac{d}{dt}\|\langle x\rangle^\alpha\langle y\rangle^\beta\mu(t)\|_{L^2_{x,y}} \lesssim \| \gamma \|_{X_{T}}^{2} \langle t\rangle^{-1+\delta} 
\]
for any $|\alpha|, |\beta|\le 2$. 
\end{lemma}
\begin{proof}
Recall that $\mu(t,x,y) = e^{-it(\Delta_x - \Delta_y)} \gamma(t,x,y)$, and so 
\begin{equation*}
	i \partial_{t} \mu(t,x,y) =(W(t,x) - W^*(t,y))\mu(t,x,y)
\end{equation*}
where $W(t,x) = e^{-it\Delta_x }V(t,x) e^{it\Delta_x}$. Multiplying the equation by $\langle x\rangle^\alpha\langle y\rangle^\beta$, we obtain  
$$
\begin{aligned}
\frac{d}{dt}\|\langle x\rangle^\alpha\langle y\rangle^\beta\mu(t)\|_{L^2_{x,y}} 
&\le \|\langle x\rangle^\alpha\langle y\rangle^\beta(W(t,x) - W^*(t,y))\mu(t)\|_{L^2_{x,y}}
\\
&\lesssim \|\langle x\rangle^\alpha W(t,x) \langle y\rangle^\beta\mu(t)\|_{L^2_{x,y}} + \|\langle x\rangle^\beta W(t,x) \langle y\rangle^\alpha\mu(t)\|_{L^2_{x,y}}.
\end{aligned}$$
For $\alpha  =\beta =0$, we use the unitary properties of $e^{\pm it \Delta_x}$ on $L^2$ to bound 
$$ \|W(t,x) \mu(t)\|_{L^2_{x,y}} \le \| V(t)\|_{L^\infty} \| \mu(t)\|_{L^2_{x,y}} \lesssim \| \gamma \|_{X_{T}}^{2} \langle t\rangle^{-1+\delta}.$$ 
On the other hand, for $|\alpha| \not =0$, we observe that 
$$ [x,e^{ \pm it \Delta_x}] = \mp 2it e^{\pm it \Delta_x}\nabla_x .$$
This yields 
$$
\begin{aligned}
[x,W] &= [x,e^{-it\Delta_x }V(t,x) e^{it\Delta_x}]
\\
&= [x,e^{-it\Delta_x }] V(t,x) e^{it\Delta_x} + e^{-it\Delta_x }V(t,x) [x,e^{it\Delta_x}]
\\
&=2it e^{- it \Delta_x}(\nabla_x V(t,x)) e^{it\Delta_x}.
\end{aligned}$$
Similarly, we compute 
$$
\begin{aligned}
[|x|^2,W] &= x [x,W] + [x,W]x = [x, [x,W]] + 2[x,W]x
\\
&=2it [ x , e^{- it \Delta_x}(\nabla_{x} V(t,x)) e^{it\Delta_x}] + 4it e^{- it \Delta_x}(\nabla_{x} V(t,x)) e^{it\Delta_x} x
\\
&= - 4 t^2 e^{- it \Delta_x}(\Delta_x V(t,x)) e^{it\Delta_x} + 4it e^{- it \Delta_x}(\nabla_{x} V(t,x)) e^{it\Delta_x} x 
\\
&=  \pm 4 t^2 e^{- it \Delta_x} \rho(t,x)e^{it\Delta_x} + 4it e^{- it \Delta_x}(\nabla_x V(t,x)) e^{it\Delta_x} x,
\end{aligned}$$
in which we have used $-\Delta_x V = \pm \rho$. Therefore, we have 
$$
\begin{aligned}
\frac{d}{dt}\|x\mu(t)\|_{L^2_{x,y}} 
&\lesssim \|x W(t,x)\mu(t)\|_{L^2_{x,y}} +  \|W(t,x)y\mu(t)\|_{L^2_{x,y}}
\\
&\lesssim t \|e^{- it \Delta_x}(\nabla_x V(t,x)) e^{it\Delta_x}\mu(t)\|_{L^2_{x,y}} +  \|W(t,x)(|x|+|y|)\mu(t)\|_{L^2_{x,y}}
\\
&\lesssim t \|\nabla_x V(t)\|_{L^\infty_x} \|\mu(t)\|_{L^2_{x,y}} + \|V(t)\|_{L^\infty_x} \|(|x|+|y|)\mu(t)\|_{L^2_{x,y}},
\end{aligned}$$
and $$
\begin{aligned}
\frac{d}{dt}\|x^2\mu(t)\|_{L^2_{x,y}} 
&\lesssim  \|W(t,x)(|x|^2+ |y|^2)\mu(t)\|_{L^2_{x,y}} + t \|e^{- it \Delta_x}(\nabla_x V(t)) e^{it\Delta_x}(|x|+|y|)\mu(t)\|_{L^2_{x,y}} 
\\&\quad +  t^2 \|e^{- it \Delta_x}\rho(t,x)e^{it\Delta_x}\mu(t)\|_{L^2_{x,y}}
\\&\lesssim ( \|V(t)\|_{L^\infty_x} + t\|\nabla_xV(t)\|_{L^\infty_x} )\|(1+|x|^2+ |y|^2)\mu(t)\|_{L^2_{x,y}} +  t^2 \|\rho(t)\|_{L^\infty_x}\|\mu(t)\|_{L^2_{x,y}},
\end{aligned}$$
all of which are again bounded by $\| \gamma \|_{X_{T}}^{2} \langle t\rangle^{-1+\delta}$, upon using Lemma \ref{lem-Vde}. The weight $\langle y\rangle^\beta$ is treated similarly. 
\end{proof}

\section{Scattering ansatz}\label{subsec:linfty}

In view of Lemma \ref{lemma:linfty}, it remains to bound $\widehat{\mu}(t)$ in $L^\infty_{k,p}$ with $\mu = e^{-it \Delta} \gamma e^{it \Delta}$. Indeed, taking the Fourier transform of \eqref{eqs-gamma} with respect to $x$ and $y$, with dual variables $k$ and $p$, we get
\begin{equation*}
	i \partial_{t} \widehat{\gamma}(t,k,p) = (|k|^2 - |p|^2) \widehat{\gamma}(t,k,p) + \int_{\RR^3} \widehat{V}(t,\ell) ( \widehat{\gamma}(t, k-\ell, p) - \widehat{\gamma}(t, k, p-\ell)) \, d\ell ,
\end{equation*}
in which $ \widehat{V}(t,\ell) = \pm |\ell|^{-2}\widehat\rho(t,\ell)$. By definition, we have $
 	\widehat{\mu}(t,k,p)= e^{it(|k|^2-|p|^2)} \widehat{\gamma}(t, k,p),
$
and so 
\begin{equation}\label{eqs-muhat}
	i \partial_{t}\widehat{\mu}(t,k,p) 
	= \int_{\RR^3}\widehat{V}(t,\ell) \left(e^{it\ell \cdot (2k-\ell)} \widehat{\mu}(t, k-\ell, p) - e^{-it\ell \cdot(2p-\ell)}\widehat{\mu}(t, k, p-\ell) \right) \, d\ell.
\end{equation}
We first analyze the integral term on the right hand side. In view of Lemma \ref{lem-Vde}, the potential function $V(t,x)$ decays exactly at rate $1/t$, which is not integrable in time. However, similarly as classically done for Schr\"odinger and Hartree equations (for the rank-one case), e.g., see \cite{Kato2011}, the $1/t$ decay term requires a phase correction. To extract this, we write 
$$ e^{it\ell \cdot (2k-\ell)} \widehat{\mu}(t, k-\ell, p) = e^{2it\ell \cdot k } \widehat{\mu}(t, k, p) + e^{2it\ell \cdot k} ( e^{-it |\ell|^2}\widehat{\mu}(t, k-\ell, p)  - \widehat{\mu}(t, k, p)  ).$$
Similarly done for $e^{-it\ell \cdot(2p-\ell)}\widehat{\mu}(t, k, p-\ell)$. This leads to 
$$
\begin{aligned}
	i \partial_{t}\widehat{\mu}(t,k,p) 
	&=  \widehat{\mu}(t, k, p)  \int_{\RR^3} \widehat{V}(t,\ell) \left(e^{2it\ell \cdot k} - e^{-2it\ell \cdot p}\right)  d\ell
	\\
	&\quad + \int_{\RR^3} \widehat{V}(t,\ell) e^{2it\ell \cdot k}  \left(e^{-it|\ell|^2} \widehat{\mu}(t, k-\ell, p) - \widehat{\mu}(t, k, p) \right)  d\ell
		\\
	&\quad - \int_{\RR^3}\widehat{V}(t,\ell) e^{- 2it\ell \cdot p} \left(e^{it|\ell |^2}\widehat{\mu}(t, k, p-\ell) - \widehat{\mu}(t, k, p)  \right)  d\ell.
\end{aligned}
$$
Note that the first integral term is equal to $V(t,2tk) - V(t,-2tp)$, which is a real-valued function, recalling $V(t,x) = \pm |x|^{-1}\star_x \rho(t,x)$. Therefore, introduce the phase correction integral 
\begin{equation}\label{def-Ephase}
	\Phi(t,k,p) = \Psi(t,k) - \Psi(t,-p),  \qquad \Psi(t,k) =\int_{0}^{t} V(s,2sk)\, ds, 
\end{equation}
for $k,p \in \RR^3$, and set 
\begin{equation}\label{def-newnu}\widehat{\nu}(t,k,p)=  e^{i\Phi(t,k,p)} \widehat{\mu}(t,k,p) = e^{i\Psi(t,k)}  e^{it|k|^2} \widehat{\gamma}(t, k,p)e^{-i\Psi(t,-p)}  e^{-it|p|^2} .
\end{equation}

We obtain the following lemma.

\begin{lemma}\label{lemma:linfty-mod} Introduce $\widehat\nu(t,k,p)$ as in \eqref{def-Ephase}-\eqref{def-newnu}. Then, for any $0\le t\le T$, there holds
	\[
	\frac{d}{dt} |\widehat{\nu}(t,k,p)| \lesssim \| \gamma \|_{X_{T}}^{2} \langle t\rangle^{-5/4 + 2\delta},	\]
	uniformly in $k,p \in \RR^3$. In particular, $\|  \widehat{\mu}(t)\|_{L^\infty_{k,p}} \lesssim  \varepsilon+\| \gamma \|_{X_{T}}^{2}$. 
\end{lemma}

\begin{remark}
	In view of the proof of Lemma \ref{lemma:linfty}, it suffices for the global well-posedness theory to estimate $\sup_{k} |\widehat{\mu}(t,k,-k)|$, for which the phase correction is not necessary since $\Phi(t,k,-k) =0$.
\end{remark}

\begin{proof}[Proof of Lemma \ref{lemma:linfty-mod}]
By construction, we compute 
\begin{equation}\label{eqs-nuhat}
	 i\partial_{t} \widehat{\nu}(t,k,p) = e^{i\Phi(t,k,p)} \Big( R_1(t,k,p) +  R_2(t,k,p)\Big)
	\end{equation}
in which 
$$
\begin{aligned}
	R_{1} (t,k,p)
	&= \int_{\RR^3} \widehat{V}(t,\ell) e^{2it\ell \cdot k}  \left(e^{-it|\ell|^2} \widehat{\mu}(t, k-\ell, p) - \widehat{\mu}(t, k, p) \right)  d\ell
		\\
	R_2(t,k,p)&= -\int_{\RR^3}\widehat{V}(t,\ell) e^{- 2it\ell \cdot p} \left(e^{it|\ell |^2}\widehat{\mu}(t, k, p-\ell) - \widehat{\mu}(t, k, p)  \right)  d\ell .
	\end{aligned}
$$
We now bound these integral terms. We focus on estimating $R_1(t,k,p)$; the integral $R_2(t,k,p)$ is similar. 
We first bound the integral over the region where $|\ell |\ge \langle t\rangle^{-1/2}$. Indeed, recalling $\widehat{V}(t,\ell) = \pm |\ell|^{-2}\widehat \rho(t,\ell)$, we bound
\begin{align}
	\label{eq:highfreq}
	\begin{split}
\Big| \int_{|\ell| \ge \langle t \rangle^{-1/2}} &\widehat{V}(t,\ell) e^{2it\ell \cdot k}  \left(e^{-it|\ell|^2} \widehat{\mu}(t, k-\ell, p) - \widehat{\mu}(t, k, p) \right)  d\ell\Big|
	\\& \le 2 \| \widehat{\mu}(t)\|_{L^{\infty}_{k,p}} \, \int_{|\ell| \ge \langle t \rangle^{-1/2}} |\ell|^{-2}| \widehat{\rho}(t,\ell)| \,d\ell 
	\\
	&\lesssim  \| \widehat{\mu}(t)\|_{L^{\infty}_{k,p}} \, \| \widehat\rho(t) \|_{L^2_\ell} \, \| |\ell|^{-2} \chi_{ |\ell| \ge \langle t \rangle^{-1/2}} \|_{L^2_\ell} 
		\\
	&\lesssim  \langle t\rangle^{1/4}\| \widehat{\mu}(t)\|_{L^{\infty}_{k,p}} \, \|\rho(t) \|_{L^2_x} 
		\\
	&\lesssim  \langle t\rangle^{-5/4 + \delta} \| \gamma \|^2_{X_{T}}
	\end{split}
\end{align}
in which we used the bootstrap assumption $\| \widehat{\mu}(t)\|_{L^{\infty}_{k,p}} \lesssim \|\mu\|_{L^1_{x,y}} \lesssim \|\langle x\rangle^2 \langle y\rangle^2\mu\|_{L^2_{x,y}} \lesssim \langle t\rangle^\delta\| \gamma \|_{X_{T}}$ and $\|\rho(t) \|_{L^2_x} \lesssim \varepsilon \langle t\rangle^{-3/2}\| \gamma \|_{X_{T}}$. On the other hand, for $|\ell |\le \langle t\rangle^{-1/2}$, we bound 
\begin{align*}
	|e^{-it\ell^2}\widehat{\mu}(t,k-\ell,p) - \widehat{\mu}(t,k,p) | 
	&\lesssim |e^{-it\ell^2} (\widehat{\mu}(t,k-\ell,p) - \widehat{\mu}(t,k,p))| + |(e^{-it\ell^2} - 1) \widehat{\mu}(t,k,p)| 
	\\
	&\lesssim  |\ell|^{1/2} \| \widehat{\mu}(t)\|_{L^{\infty}_{p} C^{1/2}_{k}} + t|\ell|^2 \| \widehat{\mu}(t)\|_{L^\infty_{k,p}}
	\\
	&\lesssim |\ell|^{1/2} \langle t \rangle^{1/4 + \delta} \| \gamma \|_{X_{T}},
\end{align*}
in which we used $ \| \widehat{\mu}(t)\|_{L^{\infty}_{p} C^{1/2}_{k}} \lesssim  \| \widehat{\mu}(t)\|_{H^2_p H^2_k} \lesssim \| \langle x\rangle^2 \langle y\rangle^2\mu(t)\|_{L^2_{x,y}}$ and $t |\ell|^{3/2} \lesssim t^{1/4}$. Therefore, recalling again $\widehat{V}(t,\ell) = \pm |\ell|^{-2}\widehat \rho(t,\ell)$, we bound 
\begin{align}
	\begin{split}
	\label{eq:I2-I21}
\Big| \int_{|\ell| \le \langle t \rangle^{-1/2}} &\widehat{V}(t,\ell) e^{2it\ell \cdot k}  \left(e^{-it|\ell|^2} \widehat{\mu}(t, k-\ell, p) - \widehat{\mu}(t, k, p) \right) d\ell\Big|
\\	&\lesssim \langle t \rangle^{1/4 +\delta}  \| \gamma\|_{X_{T}} \int_{|\ell| \le \langle t \rangle^{-1/2}} |\ell|^{-3/2} |\widehat{\rho}(t,\ell)| \,d\ell \\
	&\lesssim \langle t \rangle^{1/4 +\delta} \| \gamma \|_{X_{T}} \,  \| \rho(t) \|_{L^p_{x}} \, \| |\ell|^{-3/2} \chi_{|\ell| \le \langle t \rangle^{-1/2}} \|_{L^p_{\ell}} \\
	&\lesssim \langle t \rangle^{- 5/4 + \delta + 3\delta'/2} \| \gamma \|_{X_{T}}^2,
	\end{split}
\end{align}
for $p = (1/2 + \delta')^{-1}$, where $\delta' >0$ is fixed. Combining \eqref{eq:highfreq} and \eqref{eq:I2-I21}, we have obtained 
$$| R_{1} (t,k,p)| \lesssim  \langle t\rangle^{-5/4 + \delta} \| \gamma \|^2_{X_{T}} +  \langle t \rangle^{- 5/4 + \delta + 3\delta'/2} \| \gamma \|_{X_{T}}^2.$$
The lemma thus follows, upon choosing $\delta' = 2\delta/3$ (and $\delta$ as in the iterative norm \eqref{def-normX} is chosen to be  sufficiently small). 
\end{proof}

\begin{lemma}\label{lemma:l2-mod} Introduce $\widehat\nu(t,k,p)$ as in \eqref{def-Ephase}-\eqref{def-newnu}. Then, for any $0\le t\le T$, there holds
	\[
	\frac{d}{dt} \|\nu(t)\|_{L^2_{x,y}} \lesssim \| \gamma \|_{X_{T}}^{2} \langle t\rangle^{-5/4 +\delta}.	\]
\end{lemma}

\begin{proof} The proof follows similarly as done in the previous lemma. Indeed, it suffices to bound the integral terms $R_1(t,k,p)$ in $L^2_{k,p}$. We first bound 
$$
\begin{aligned}
\Big\| \int_{|\ell| \ge \langle t \rangle^{-1/2}} &\widehat{V}(t,\ell) e^{2it\ell \cdot k}  \left(e^{-it|\ell|^2} \widehat{\mu}(t, k-\ell, p) - \widehat{\mu}(t, k, p) \right)  d\ell\Big\|_{L^2_{k,p}}
	\\& \le 2 \| \widehat{\mu}(t)\|_{L^{2}_{k,p}} \, \int_{|\ell| \ge \langle t \rangle^{-1/2}} |\ell|^{-2} | \widehat{\rho}(t,\ell)| \,d\ell 
\end{aligned}
$$
which is again bounded by  $\langle t\rangle^{-5/4 + \delta} \| \gamma \|^2_{X_{T}}$, exactly as done in \eqref{eq:highfreq}. 
On the other hand, for $|\ell |\le \langle t\rangle^{-1/2}$, we bound 
$$ 
|\widehat{\mu}(t,k-\ell,p) - \widehat{\mu}(t,k,p)|^2 \le \Big|\int_0^1 \ell \cdot \nabla_k \widehat\mu(t,k-\theta\ell, p)\, d\theta\Big|^2  \le |\ell|^2\int_0^1 |\nabla_k \widehat\mu(t,k-\theta\ell, p)|^2 \, d\theta.
$$

This yields 
\begin{align*}
	\|e^{-it\ell^2}\widehat{\mu}(t,k-\ell,p) - \widehat{\mu}(t,k,p) \|_{L^2_{k,p}} 
	&\lesssim \|\widehat{\mu}(t,k-\ell,p) - \widehat{\mu}(t,k,p)\|_{L^2_{k,p}} + |(e^{-it\ell^2} - 1)|  \|\widehat{\mu}(t)\|_{L^2_{k,p}} 
	\\
	&\lesssim |\ell|\|\nabla_k \widehat\mu(t)\|_{L^2_{k,p}}+ t|\ell|^2 \| \widehat{\mu}(t)\|_{L^2_{k,p}} 
	\\
	&\lesssim |\ell|\langle t \rangle^{1/2} \| \gamma \|_{X_{T}},
\end{align*}
recalling $t|\ell|\le \langle t\rangle^{1/2}$. 
Therefore, we bound 
$$
\begin{aligned}
\Big\| \int_{|\ell| \le \langle t \rangle^{-1/2}} &\widehat{V}(t,\ell) e^{2it\ell \cdot k}  \left(e^{-it|\ell|^2} \widehat{\mu}(t, k-\ell, p) - \widehat{\mu}(t, k, p) \right)  d\ell \Big\|_{L^2_{k,p}}
\\  &\lesssim \langle t \rangle^{1/2}  \| \gamma\|_{X_{T}} \int_{|\ell| \le \langle t \rangle^{-1/2}} |\ell|^{-1} |\widehat{\rho}(t,\ell)| \,d\ell \\
	&\lesssim \langle t \rangle^{1/2} \| \gamma \|_{X_{T}} \,  \| \rho(t) \|_{L^2_{x}} \, \| |\ell|^{-1} \chi_{|\ell| \le \langle t \rangle^{-1/2}} \|_{L^2_{\ell}} \\
	&\lesssim \langle t \rangle^{- 5/4 + \delta} \| \gamma \|_{X_{T}}^2.
\end{aligned}
$$
Hence we obtain the lemma. 
\end{proof}

\section{Proof of Theorem \ref{thm:GWP}}

In this section, we complete the proof of Theorem \ref{thm:GWP}. Indeed, introduce the iterative norm $\|\gamma\|_{X_T}$ as in \eqref{def-normX}, 
which we recall 
$$
	\| \gamma\|_{X_t}= \sup_{0\le s\le t} \Big\{ \|\rho(s)\|_{L^1}+
\langle s \rangle^{3} \| \rho(s)\|_{L^{\infty}} 
		+ \langle s \rangle^{-\delta} \| \gamma(s) \|_{\cH^{2}}
		+ \langle s \rangle^{-\delta} \|\mu(s) \|_{\cA^2}
	\Big\},
$$
in which $\mu(t) = e^{-it \Delta} \gamma(t) e^{it \Delta}$. We now establish the a priori bounds \eqref{apriori}: namely, $\| \gamma\|_{X_t} \le C_0 \varepsilon  + C_1 \| \gamma\|_{X_t}^2$. The boundedness of $\rho(t)$ in $L^1_x$ follows from the conservation of $\gamma(t)$ in the finite trace class, while using Lemma \ref{lemma:linfty} for any $0<\beta <1/2$, we bound 
$$ \|\rho(t)\|_{L^\infty} \lesssim \langle t\rangle^{-3} \| \widehat\mu(t)\|_{L^\infty_{k,p}} + \langle t\rangle^{-3-\beta} \|\langle x\rangle^2 \langle y\rangle^2\mu(t)\|_{L^2_{x,y}}.$$
Thanks to Lemmas \ref{lemma:linfty-mod} and \ref{lem-bdmu}, we have $\|  \widehat{\mu}(t)\|_{L^\infty_{k,p}} \lesssim  \varepsilon+\| \gamma \|_{X_{T}}^{2}$ and $\|\langle x\rangle^2 \langle y\rangle^2\mu(t)\|_{L^2_{x,y}} \lesssim \varepsilon+\langle t\rangle^\delta\| \gamma \|^2_{X_{T}}$, respectively.
Choosing $\beta \ge\delta$, the above thus yields 
$$\langle t\rangle^3\|\rho(t)\|_{L^\infty} \lesssim \varepsilon+\| \gamma \|_{X_{T}}^{2}.$$
Bounds on $\| \gamma(t) \|_{\cH^{2}}$ and $\|\mu(t) \|_{\cA^2}$ are already obtained in Lemmas \ref{lem-EE} and \ref{lem-bdmu}, respectively, completing the proof of \eqref{apriori}. This proves that $\|\gamma\|_{X_t} \lesssim \varepsilon$ for all positive times $t\ge 0$. 

It remains to establish the modified scattering of solutions $\gamma(t)$. Indeed, in view of Lemma \ref{lemma:l2-mod}, we obtain 
\[
	\|\nu(t_{1}) - \nu(t_{2})\|_{L^2_{x,y}} \lesssim \varepsilon ^2  \int_{t_{2}}^{t_{1}} \langle s\rangle^{-5/4+\delta} \, ds \lesssim \varepsilon^2 \langle t_2\rangle^{-1/4 + \delta},
\]
for all $t_{1} \ge t_{2} \ge 0$ so that $\nu(t)$ has the unique limit $\gamma_{\infty} \in L^2_{x,y}$ as $t\to \infty$. In addition, there holds
\[
	\| \nu(t) - \gamma_{\infty} \|_{L^2_{x,y}} \lesssim \varepsilon  \langle t \rangle^{-1/4 + \delta}.
\]
Recalling \eqref{def-newnu}, this proves that 
\begin{equation}\label{cv-gammam}
	\| \gamma(t) - e^{-i\Psi(t,-i\nabla_x) }  e^{it\Delta_x} \gamma_{\infty} e^{i\Psi(t,-i\nabla_x) }  e^{-it\Delta_x} \|_{L^2_{x,y}} \lesssim \varepsilon  \langle t \rangle^{-1/4 + \delta},
\end{equation}
in which the phase modification $\Psi(t,k)$ is defined as in \eqref{def-Ephase}. Let us further expand $\Psi(t,k)$. In view of the proof of Lemma \ref{lemma:linfty}, see \eqref{den-ld}, for $t\ge 1$, we may write 
\begin{align*}
	\rho(t,x) = \frac{1}{t^3} \widehat{\gamma_{\infty}} \left( \frac{x}{2t}, - \frac{x}{2t} \right) + \rho^r(t,x),  
\end{align*}
in which $|\rho^r(t,x)|\lesssim \varepsilon  \langle t\rangle^{-3-\delta}$.
For $0\le t\le 1$, we simply bound $\rho(t,x)$ by $C_0 \varepsilon$. 
Hence, we compute 
\begin{align*}
	\Psi(t,k)
	&= \int_0^{t} V(s, 2sk) ds =\pm \int_0^{t} \int_{\mathbb{R}^3} \frac{1}{|2sk- \eta|} \rho(s, \eta) \,d\eta ds \\
	&= \pm \int_0^{t} 4s^2 \int_{\mathbb{R}^3}\frac{1}{|k-\eta|} \rho(s,2s\eta) \, d\eta ds \\
	&= \pm 4 \log t \int_{\mathbb{R}^3}\frac{1}{|k-\eta|}  \widehat{\gamma_{\infty}} (\eta, -\eta) d\eta + \int_0^1 V(s, 2sk) ds \pm \int_1^{t} |x|^{-1}\star_x\rho^r(s,2sk) \; ds,
\end{align*}
in which the integrand in the last integral satisfies $|x|^{-1}\star_x\rho^r(s,2sk) \lesssim \varepsilon \langle s\rangle^{-1-\delta}$. Set 
	$$
	\begin{aligned}
	 g_\infty(k) &=  \pm \int_{\mathbb{R}^3}\frac{1}{|k-\eta|}  \widehat{\gamma_{\infty}} (\eta, -\eta) d\eta
	 \\
	 h_\infty(k) &= \int_0^1 V(s, 2sk) ds \pm \int_1^\infty |x|^{-1}\star_x\rho^r(s,2sk) \; ds
	 \end{aligned}$$
both of which are stationary. Hence, we obtain 
$$	\Psi(t,k) =  h_\infty(k) + g_\infty(k) \log t + \Psi^r(t,k)$$
where $ \Psi^r(t,k) = \mp \int_t^\infty |x|^{-1}\star_x\rho^r(s,2sk) \; ds$, which is bounded by $\varepsilon \langle t\rangle^{-\delta}$. Putting these into \eqref{cv-gammam}, we deduce \eqref{modifieds} as claimed, upon redefining $\gamma_\infty$ by $e^{-i h_\infty(-i\nabla_x)} \gamma_\infty e^{i h_\infty(-i\nabla_x)}$ and replacing $e^{\pm i\Psi^r(t,-i\nabla_x)}$ by $1$ leading to an error of order $\varepsilon t^{-\delta}$. This ends the proof of Theorem \ref{thm:GWP}. 

~\\
{\em Acknowledgement.}
The authors would like to thank Phan Thanh Nam for his many fruitful discussions related to this work. The research is supported in part by the NSF under grants DMS-2054726 and DMS-2349981.

\bibliographystyle{plain}

\end{document}